\documentclass[proceedings,submission%
]{dmtcs} 

\usepackage[utf8]{inputenc}
\usepackage{subfigure}

\usepackage{times}
\usepackage{amssymb,amsmath}
\usepackage{xspace}
\usepackage{epsfig}
\usepackage{enumitem}
\usepackage{xcolor}
\usepackage{tikz}
\usepackage{gensymb}
\usepackage[T1]{fontenc}
\usepackage{bm}

\usepackage{hyperref}

\newtheorem{lemma}{Lemma}[section]
\newtheorem{theorem}[lemma]{Theorem}

\newtheorem{proposition}[lemma]{Proposition}

\newtheorem{defprop}[lemma]{Definition-Proposition}

\newtheorem{remark}{Remark}
\newtheorem{fact}[remark]{Fact}

\newcommand{\RR}{\mathbb{R}}

\newcommand{\II}{\mathbb{I}}
\newcommand{\OO}{\mathbb{O}}
\newcommand{\QQ}{\mathbb{Q}}
\newcommand{\PP}{\mathbb{P}}
\newcommand{\PPna}{\PP_n^{(\alpha)}}

\newcommand{\Pola}{\Lambda_\star^{(\alpha)}}
\newcommand{\Pol}{\Lambda_\star}

\newcommand{\esper}{\mathbb{E}}

\newcommand{\comment}[1]{}

\DeclareMathOperator{\Sym}{Sym}

\DeclareMathOperator{\Var}{Var}

\def\Ch{Ch}

\DeclareMathOperator{\Vect}{Vect}

\author[M.~Do\l\k ega and V. F\'eray]{Maciej Do\l\k ega\addressmark{1}\thanks{
Research of MD is supported by the grant of National Center of Sciences 2011/03/N/ST1/00117 for the years 2012--2014 and by the contract with Uniwersytet
Wrocławski 1380/M/IM/11.}
\and Valentin F\'eray\addressmark{2}\thanks{Partially supported by ANR project PSYCO.}}

\address{\addressmark{1}Instytut Matematyczny,
Uniwersytet Wroc\l awski,  \mbox{pl.\ Grunwaldzki~2/4,} 50-384
Wroc\l aw, Poland\\
\addressmark{2}LaBRI, Universit\'e Bordeaux 1, 351 cours de la Lib\'eration, 33 400
 Talence, France}


\title{On Kerov polynomials for Jack characters}

\begin{document}

\maketitle

\begin{abstract}
    {\bf Abstract.}
    We consider a deformation of Kerov character polynomials,
    linked to Jack symmetric functions.
    It has been introduced recently by M.~Lassalle, who formulated several
    conjectures on these objects, suggesting some underlying combinatorics.
    We give a partial result in this direction, showing
    that some quantities are polynomials in the Jack parameter $\alpha$
    with prescribed degree.

    Our result has several interesting consequences in various directions.
    Firstly, we give a new proof of the fact that the coefficients of Jack 
    polynomials expanded in the monomial or power-sum basis depend
    polynomially in $\alpha$.
    Secondly, we describe asymptotically the shape of random Young diagrams
    under some deformation of Plancherel measure.

    {\bf R\'esum\'e.}
    On considère une déformation des polynômes de Kerov pour les caractères
    du groupe symétrique.
    Cette déformation est liée aux polynômes de Jack.
    Elle a été récemment définie par M.~Lassalle, qui a proposé plusieurs
    conjectures sur ces objets, suggérant ainsi l'existence d'une
    combinatoire sous-jacente.
    Nous donnons un résultat partiel dans cette direction,
    en montrant que certaines quantités sont des polynômes
    (dont on contrôle les degrés)
    en fonction du paramètre de Jack $\alpha$.

    Notre résultat a des conséquences intéressantes dans des directions
    diverses.
    Premièrement, nous donnons une nouvelle preuve de la polynomialité 
    (toujours en fonction de $\alpha$) des
    coefficients du développement des polynômes de Jack dans la base monomiale.
    Deuxièmement, nous décrivons asymptotiquement la forme
    de grands diagrammes
    de Young distribués selon une déformation de la mesure de Plancherel.
    \end{abstract}

This paper is an extended abstract of \cite{FullPaper},
which will be submitted elsewhere.

\section{Introduction}

\subsection{Polynomiality of Jack polynomials}
In a seminal paper \cite{Jack1970/1971},
H.~Jack introduced a family of symmetric functions
$J^{(\alpha)}_\lambda$ depending on an additional parameter
$\alpha$ called \emph{Jack polynomials}.
Up to multiplicative constants, for $\alpha=1$, Jack polynomials coincide with Schur polynomials.
Over the time, it has been shown that several
results concerning Schur polynomials can be generalized in a rather
natural way to Jack polynomials (Section (VI,10) of I.G.~Macdonald's book \cite{Macdonald1995}
gives a few results of this kind).\pagebreak

One of the most surprising features of Jack polynomials is that
they have several equivalent classical definitions,
but none of them makes obvious
the fact that the coefficients of their expansion on the monomial basis
are polynomials in $\alpha$
(by construction, they are only rational functions).
%
This property has been established by Lapointe and Vinet \cite{LapointeVinetJack}.
One of the result of this paper is a new proof of
Lapointe-Vinet theorem.
\begin{theorem}[Lapointe and Vinet \cite{LapointeVinetJack}]
    The coefficients of the expansion of Jack polynomials in the monomial
    basis are polynomials in $\alpha$.
    \label{ThmCoefJackPoly}
\end{theorem}
This theorem is proved in Section \ref{SubsectCoefJackPoly}.
We believe that this new proof is interesting in itself,
because it relies on a very different approach to Jack polynomials.

To be comprehensive on the subject, let us mention that the coefficients
of these polynomials are in fact non-negative integers.
This result had been conjectured by R. Stanley and I. Macdonald;
see {\em e.g.} \cite[VI, equation (10.26?)]{Macdonald1995}.
It was proved by Knop and Sahi \cite{KnopSahiCombinatoricsJack},
shortly after Lapointe-Vinet's paper.
We are unfortunately unable to prove this stronger result with our methods.

\subsection{Dual approach}
We will later define \emph{Jack character} to be equal (up to some simple normalization constant) 
to the coefficient $[p_\mu] J_\lambda$ in the expansion of the Jack polynomial $J_\lambda$ in the basis of power-sum symmetric functions.
The idea of the {\em dual approach} is to consider Jack characters
as a function of $\lambda$ and not as a function of $\mu$ as usual.
In more concrete words, we would like to express
the Jack character as a sum of some quantities depending on $\lambda$ over some
combinatorial set depending on $\mu$
(in Knop-Sahi's result, it is roughly the opposite). 

Inspired by the case $\alpha=1$ (which corresponds to the usual characters of the symmetric groups), 
Lassalle \cite{Lassalle2009} suggested  
to express Jack characters in terms of, so called, free cumulants
of the transition measure of the Young diagram $\lambda$.
This expression, called \emph{Kerov polynomials for Jack characters},
involves rational functions in $\alpha$, which are
conjecturally polynomials with non-negative integer coefficients in $\alpha$ and
$\beta=1-\alpha$ (see 
\cite[Conjecture 1.2]{Lassalle2009});
we refer to this as Lassalle's conjecture.
This suggests the existence of a combinatorial interpretation.
A result of this type holds true in the case $\alpha=1$,
see \cite{NousKerovExplicitInterpretation}.

In this paper, we prove a part of Lassalle's conjecture,
that is the polynomiality in $\alpha$
(but neither the non-negativity, nor the integrity) of the coefficients.

\begin{theorem}
The coefficients of Kerov polynomials for Jack characters are polynomials
in $\alpha$ with rational coefficients.
    \label{ThmPol}
\end{theorem}
This theorem with a precise bound on the degree of these polynomials
is stated in Section \ref{SubsectMainResult}.
In this extended abstract, we only give the guidelines of the proof.

\subsection{Applications}
Our bounds for degrees of coefficients of 
Kerov polynomials for Jack characters imply in particular
that some coefficients (corresponding to the leading term for some gradation)
are independent on $\alpha$.
In Section \ref{SectRandomYD}, we use this simple remark to describe asymptotically
the shape of random Young diagrams whose distribution is a deformation of Plancherel
measure.

Another consequence of our results
is a uniform proof of the polynomiality
of structure constants of several meaningful algebras.
This allows us to solve some conjectures of Matsumoto \cite{MatsumotoOddJM}
and to give a partial answer to the Matching-Jack conjecture
of I. Goulden and D. Jackson \cite{GouldenJacksonMatchingJack}.
Due to the lack of space, we will not present these results in this
extended abstract.
They can be found in \cite[Section 4]{FullPaper}.\pagebreak


{\em Outline of the paper.}
Section \ref{SectDefKerov} gives all
necessary definitions and background; in particular we recall the notions of
free cumulants and Kerov polynomials.
In Section \ref{SectPolynomial} we sketch the proof of
Theorem \ref{ThmPol} and Theorem \ref{ThmCoefJackPoly}.
Finally, we state and prove our results on large Young diagrams in Section
\ref{SectRandomYD}.


\section{Jack characters and Kerov polynomials}
\label{SectDefKerov}

\subsection{Partitions and symmetric functions}
We begin with a few classical definitions and notations.

A {\em partition} $\lambda$ of $n$ (denote it by $\lambda \vdash n$)
is a non-increasing
list $(\lambda_1,\dots,\lambda_\ell)$ of positive integers of sum equal to $n$.
Then $n$ is called the {\em size} of $\lambda$ and denoted by $|\lambda|$
and the number $\ell$ is the {\em length} of the partition (denoted by $\ell(\lambda)$).

We also consider the graded ring of symmetric functions $\Sym$.
Recall that its homogeneous component $\Sym_n$ 
of degree $n$ admits several classical
bases: the monomials $(m_\lambda)_{\lambda \vdash n}$,
the power-sums $(p_\lambda)_{\lambda \vdash n}$, each indexed
by partitions of $n$.
All the definitions can be found in \cite[Chapter I]{Macdonald1995}.

Jack polynomials are symmetric functions indexed by partitions and depending on a parameter $(\alpha)$.
There exist several normalizations for Jack polynomials in the literature.
We shall work with the one denoted by $J$ in the book of   
Macdonald \cite[VI, (10.22)]{Macdonald1995} and use the same notation as he does.
For a fixed value of the parameter $\alpha$, the family 
$(J^{(\alpha)}_\lambda)_{\lambda \vdash n}$ forms a basis of $\Sym_n$.

\subsection{Jack characters} 


As power-sum symmetric functions $(p_\rho)_{\rho \vdash n}$
form a basis of $\Sym_n$,
we can expand the Jack polynomial $J_\lambda^{(\alpha)}$ in that base.
For $\lambda \vdash n$,
there exist (unique) coefficients $\theta_{\rho}^{(\alpha)}(\lambda)$ such that
\begin{equation} 
\label{eq:jack-characters}
J_\lambda^{(\alpha)}=\sum_{\substack{\rho: \\
|\rho|=|\lambda|}} 
\theta_{\rho}^{(\alpha)}(\lambda)\ p_{\rho}. 
\end{equation}
Then we can define \emph{Jack characters} by the formula:
\begin{displaymath}
\label{eq:definition-character}
\Ch_{\mu}^{(\alpha)}(\lambda)=
\alpha^{-\frac{|\mu|-\ell(\mu)}{2}}
\binom{|\lambda|-|\mu|+m_1(\mu)}{m_1(\mu)}
\ z_\mu \ \theta^{(\alpha)}_{\mu,1^{|\lambda|-|\mu|}}(\lambda),
\end{displaymath}
where $m_i(\mu)$ denotes the multiplicity of $i$ in the partition $\mu$ and
$z_\mu = \mu_1 \mu_2 \cdots \ m_1(\mu)!\ m_2(\mu)! \cdots.$

In the case $\alpha=1$, Jack polynomials correspond,
up to some normalization constants, to Schur symmetric functions.
The coefficients of the latter in the basis of the power-sum symmetric functions
are known to be equal to
the irreducible characters of the symmetric groups;
see \cite[Section I,7]{Macdonald1995}
(this explains the name {\em characters} in the general case,
even if, except for $\alpha=1/2,1,2$,
these quantities have no known representation-theoretical interpretation). 
It means that Jack characters with parameter $\alpha=1$
correspond, up to some numerical factors, 
to character values of the symmetric groups.

This normalization corresponds in fact to the one used
by Kerov and Olshanski in \cite{KerovOlshanskiPolFunc}.
These {\em normalized characters}
-- following the denomination of Kerov and Olshanski --
have plenty of interesting properties; for example
when considered as functions on the set of Young diagrams 
$\lambda\mapsto\Ch^{(1)}_{\mu}(\lambda)$, they form a linear basis (when
$\mu$ runs over the set of all partitions) of the algebra $\Lambda^\star$ of
shifted symmetric functions
, which is very rich in structure.

Jack characters have been first considered by M.~Lassalle in
\cite{Lassalle2008a}. 
Note that the normalization used here is different that the 
one of these papers.
The reason of this new choice of normalization will be clear
later.

\subsection{Generalized Young diagrams and Kerov interlacing coordinates}
\label{SubsectGeneralizedYD}

In this section, we will see different ways of representing Young diagrams and even more general objects related to them.
Let us consider a zigzag line $L$ going from a point $(0,y)$ on the $y$-axis to
a point $(x,0)$ on the $x$-axis.
We assume that every piece is either an horizontal segment from left to right or
a vertical segment from top to bottom.
A Young diagram can be seen as such a zigzag line: just consider its border.
Therefore, we call these zigzag lines \emph{generalized Young diagrams}.
\begin{figure}[tb]
    \begin{minipage}[b]{.6\linewidth}
    \[ \begin{array}{c}
        \begin{tikzpicture}[scale=0.6]
            \draw[dashed,gray] grid (4.3,3.4);
            \draw[->,thick] (-0.2,0) -- (4.7,0);
            \draw[->,thick] (0,-0.2) -- (0,3.7);
            \draw[ultra thick] (0,2) 
            -- (.5,2) 
            -- (.5,.5) 
            -- (3,.5) 
            -- (3,0) ;
        \end{tikzpicture}\\
        \lambda
    \end{array}
    \mapsto 
     \begin{array}{c}
        \begin{tikzpicture}[scale=0.6]

            \draw[dashed,gray] grid (7.2,3.5);
            \draw[->,thick] (-0.2,0) -- (7.7,0);
            \draw[->,thick] (0,-0.2) -- (0,3.7);

            \begin{scope}[xscale=2,yscale=.5]
                        \draw[ultra thick] (0,2) 
                        -- (.5,2) 
                        -- (.5,.5) 
                        -- (3,.5) 
                        -- (3,0) ;
            \end{scope}
        \end{tikzpicture}\\
        T_{2,\frac{1}{2}} (\lambda)
    \end{array}\]
    \caption{Example of stretched Young diagram.}
    \label{FigStretching}
\end{minipage}\hfill
\begin{minipage}[b]{.38\linewidth}
    \begin{center}
        \begin{tikzpicture}[scale=.6]
            \draw[dashed,gray] grid  (4.5,3.5);
            \draw[->,thick] (-.2,0) -- (4.7,0);
            \draw[->,thick] (0,-.2) -- (0,3.7);
            \draw (0,2)   node[above left,fill=white]   {\tiny $o_1=-2$};  
            \draw (.5,2)  node[above right, fill=white] {\tiny $i_1=-1.5$}; 
            \draw (.5,.5) node[above right, fill=white] {\tiny $o_2=0$};
            \draw (3,.5)  node[above right, fill=white] {\tiny $i_2=2.5$};
            \draw (3,0)   node[below right, fill=white] {\tiny $o_3=3$}; 
            \draw[ultra thick] 
                    (0,2) circle (1.5pt) --
                    (.5,2)  circle (1.5pt) --
                    (.5,.5) circle (1.5pt) --
                    (3,.5)  circle (1.5pt) --
                    (3,0)   circle (1.5pt);
        \end{tikzpicture}
    \end{center}
    \caption{A generalized Young diagram $L$ with the corresponding set $\OO_L$ and $\II_L$.}
    \label{FigGeneralizedYD}
\end{minipage}
\end{figure}

We will be in particular interested in the following generalized Young diagrams.
Let $\lambda$ be a (generalized) Young diagram and $s$ and $t$ two positive real numbers.
We denote by $T_{s,t}(\lambda)$ the broken line obtained by stretching $\lambda$
horizontally by a factor $s$ and vertically by a factor $t$ (see Figure~\ref{FigStretching}; we use french convention to draw Young diagrams).
These \emph{anisotropic} Young diagrams have been introduced by S.~Kerov in \cite{KerovAnisotropicYD}.
In the special case $s=t$, we denote by $D_s(\lambda)=T_{s,s}(\lambda)$ the
{\em dilated} Young diagram.

The content of a point of a plane is 
the difference of its $x$-coordinate and its $y$-coordinate.
We denote by $\OO_L$ the sets of contents of the \emph{outer corners} of $L$, that is
corners which are points of $L$ connecting a horizontal line on the left with vertical line on the bottom.
Similarly, the set $\II_L$ is defined as the contents of the \emph{inner corners}, that is corners 
which are points of $L$ connecting a horizontal line on the right with vertical line above.
An example is given on Figure~\ref{FigGeneralizedYD}.
The denomination {\em inner/outer} may seem strange,
but it refers to the fact that the box in the corner is inside or 
outside the diagram.

A generalized Young diagram can also be seen as a function on the real line.
Indeed, if one rotates the zigzag line counterclockwise by $45\degree$
and scale it by a factor $\sqrt{2}$ (so that the new $x$-coordinate corresponds
to contents), then it can
be seen as the graph of a piecewise affine continuous function with slope $\pm 1$.
We denote this function by $\omega(\lambda)$.
Therefore, we shall call {\em continuous Young diagram}
a Lipshitz continuous function $\omega$ with Lipshitz constant $1$
such that $\omega(x)=|x|$ for $|x|$ big enough .
This notion will be used in Section \ref{SectRandomYD}
to describe the limit shape of Young diagrams.

\subsection{Polynomial functions on the set of Young diagrams}
\label{subsec:polfunct}

If $k$ is a positive integer, one can consider the power sum symmetric function
$p_k$, evaluated on the difference of alphabets $\OO_L -\II_L$.
By definition, it is a function on generalized Young diagrams given by:
\[L \mapsto p_k(\OO_L - \II_L):= \sum_{o \in \OO_L} o^k - 
\sum_{i \in \II_L} i^k.\]
As any symmetric function can be written (uniquely) in terms of $p_k$,
we can define $f(\OO_L - \II_L)$ for any symmetric function $f$ as
follows.
Expand $f$ on the power-sum basis 
$f=\sum_{\rho} a_\rho p_{\rho_1} \cdots p_{\rho_\ell}$ 
for some family of scalars $(a_\rho)$ indexed by partitions.
Then, by definition
\[f(\OO_L - \II_L)=\sum_{\rho \text{ partition}} a_\rho p_{\rho_1}(\OO_L - \II_L)
\cdots p_{\rho_\ell} (\OO_L - \II_L).\]
This convenient notation is classical in lambda-ring calculus.

Consider the set of functions $\{\lambda \mapsto f(\OO_L - \II_L)\}$,
where $f$ describes the set of symmetric functions.
This is a subalgebra of the algebra of functions on the set of all Young diagrams.
Following S.~Kerov and G.~Olshanski, we shall call it
the algebra of {\em polynomial functions} and denote it by $\Lambda^\star$.

V.~Ivanov and G.~Olshanski \cite[Corollary 2.8]{IvanovOlshanski2002}
have shown that the normalized characters
($\lambda \mapsto \Ch^{(1)}_\mu(\lambda))_\mu$
form a linear basis of this algebra and
$(\lambda \mapsto p_k(\OO_\lambda - \II_\lambda))_{k \geq 2}$
forms an algebraic basis of $\Lambda^\star$
(for all diagrams $\lambda$, one has $p_1(\OO_\lambda - \II_\lambda)=0$).
This algebra admits several other characterization:
for instance it corresponds to the algebra of shifted symmetric functions
(see \cite[Section 1 and 2]{IvanovOlshanski2002}).\medskip

All this has a natural extension for a general parameter $\alpha$.

We say that $F$ is an \emph{$\alpha$-polynomial function} on the set of
(generalized) Young diagrams if 
    \[\lambda \mapsto F ( T_{\sqrt{\alpha}^{-1},\sqrt{\alpha}}(\lambda) ) \]
    is a polynomial function.
    The ring of $\alpha$ polynomial functions is denoted by $\Pola$.
Then $(\lambda \mapsto \Ch^{(\alpha)}_\mu(\lambda))_\mu$ forms a linear basis of $\Pola$.
This is a consequence of a result of M. Lassalle 
\cite[Proposition 2]{Lassalle2008a}.

\begin{remark}
    Lassalle's result is in fact formulated in terms of shifted symmetric functions,
    but as mentioned above, it is proved in \cite[Section 1 and 2]{IvanovOlshanski2002}
    that they correspond to polynomial functions.
\end{remark}

\begin{fact}
    With the definitions above, it should be clear that
    polynomial functions are defined on {\em generalized} Young diagrams.
    They can in fact also be canonically extended to {\em continuous}
    Young diagrams; see \cite[Section 1.2]{Biane1998}.
    This will be useful in Section \ref{SectRandomYD}.
\end{fact}

\subsection{Transition measure and free cumulants}
S.~Kerov \cite{KerovTransitionMeasure} introduced the notion
of \emph{transition measure} of a Young diagram.
This probability measure $\mu_\lambda$ associated to $\lambda$
is defined by its Cauchy transform
\[G_{\mu_\lambda}(z) = \int_\RR \frac{d\mu_\lambda(x)}{z-x} =
\frac{\prod_{i \in \II_\lambda} z-i}{\prod_{o \in \OO_\lambda} z-o}.\]
Its moments are $h_k(\OO_\lambda -\II_\lambda)$,
where $h_k$ is the complete symmetric function of degree $k$, hence they are polynomial functions on the set of Young diagrams; 
we will denote them by $M_k^{(1)}$.

In Voiculescu's free probability it is very convenient to associate to a probability measure $\mu$
a sequence of numbers $(R_k(\mu))_{k \geq 1}$ called \emph{free cumulants} 
\cite{Voiculescu1986}.
The free cumulants of the transition measure of Young diagrams appeared first in
the work of P.~Biane \cite{Biane1998} and play an important role in the asymptotic
representation theory.
As explained by M.~Lassalle (look at the case $\alpha=1$ of \cite[Section 5]{Lassalle2009}),
they can be expressed as
\[R_k^{(1)}(\lambda):= R_k(\mu_\lambda) = e_k^\star(\OO_\lambda -\II_\lambda)\]
for some homogeneous symmetric function $e_k^\star$ of degree $k$. 
Note also that $(e_k^\star)_k$ as well as complete symmetric functions $(h_k)_k$ are algebraic basis of symmetric functions and, hence
$(R_k^{(1)})_{k \geq 2}$ as well as $(M_k^{(1)})_{k \geq 2}$  are algebraic basis of ring of polynomial functions
on the set of
Young diagrams ($R_1^{(1)} = M_1^{(1)}$ is the null function).

\begin{fact}
    It is easy to see that, as $h_k$ and $e_k^\star$ are 
    homogeneous symmetric functions,
    the corresponding polynomial functions $M_k^{(1)}$ and $R_k^{(1)}$
    are compatible with dilations.
    Namely
    \[ M_k^{(1)}\big( D_s(\lambda) \big) = s^k  M_k^{(1)}\big( \lambda \big);
    \quad
     R_k^{(1)}\big( D_s(\lambda) \big) = s^k  R_k^{(1)}\big( \lambda \big).\]
\end{fact}

Using the relevant definitions,
the $\alpha$-anisotropic moments and free cumulants defined by
\begin{align*}
    M_k^{(\alpha)}(\lambda)&= M_k^{(1)} \left( T_{\sqrt{\alpha},\sqrt{\alpha}^{-1}} (\lambda) \right),\\
    R_k^{(\alpha)}(\lambda) &= R_k^{(1)} \left( T_{\sqrt{\alpha},\sqrt{\alpha}^{-1}} (\lambda) \right)
\end{align*}
are $\alpha$-polynomial and the families $(M_k^{(\alpha)})_{k\geq 2}$
and $(R_k^{(\alpha)})_{k\geq 2}$ are two algebraic basis
of the algebra $\Lambda^\star_{(\alpha)}$.

\subsection{Kerov polynomials}
Recall that Jack characters $\Ch^{(\alpha)}_\mu$ are $\alpha$-polynomial functions hence can be expressed in
terms of the two algebraic bases above.
\begin{defprop}\label{PropExistenceJackKerov}
    Let $\mu$ be a partition and $\alpha >0$ be a fixed real number.
    There exist unique polynomials $L_\mu^{(\alpha)}$ and $K_\mu^{(\alpha)}$
    such that, for every $\lambda$,
    \begin{align*}
        \Ch^{(\alpha)}_\mu(\lambda) &= L_\mu^{(\alpha)}\left(M_2^{(\alpha)}(\lambda),
            M_3^{(\alpha)}(\lambda),\cdots \right), \\
        \Ch^{(\alpha)}_\mu(\lambda) &= K_\mu^{(\alpha)}\left(R_2^{(\alpha)}(\lambda),
    R_3^{(\alpha)}(\lambda),\cdots \right).
\end{align*}
\end{defprop}
The polynomials $K_\mu^{(\alpha)}$ have been introduced by S.~Kerov in the case
$\alpha=1$ \cite{Kerov2000talk} and by M.~Lassalle in the general case
\cite{Lassalle2009}.
Once again, we emphasize that our normalizations are different from his.

From now on, when it does not create any confusion, we suppress the superscript
$(\alpha)$.

We present a few examples of polynomials $K_{\mu}$.
This data has been computed using the one given in \cite[page 2230]{Lassalle2009}
\begin{align*}
    K_{(1)} &= R_2, \\
    K_{(2)} &= R_3 + \gamma R_2, \\
    K_{(3)} &= R_4 + 3\gamma R_3 + (1 + 2\gamma^2)R_2, \\
    K_{(4)} &= R_5 + \gamma(6R_4 + R_2^2) + (5 + 11\gamma^2)R_3 + (7\gamma +
6\gamma^3)R_2, \\
K_{(2,2)} &= R_3^2 + 2\gamma R_3 R_2 - 4 R_4 + (\gamma^2-2) R_2^2 - 10 \gamma R_3
-(6\gamma^2+2) R_2.
\end{align*}

where we set $\gamma = \frac{1-\alpha}{\sqrt{\alpha}}$.
A few striking facts appear on these examples:
\begin{itemize}[topsep=1pt, partopsep=1pt, itemsep=1pt, parsep=1pt]
    \item All coefficients are polynomials in the auxiliary parameter $\gamma$:
        the sketch of the proof of this fact will be presented in the next section with explicit bounds on the 
        degrees.
    \item For one part partition, polynomials $K_{(r)}$ 
        have non-negative coefficients.
        We are unfortunately unable to prove this statement, which is a more
        precise version of \cite[Conjecture 1.2]{Lassalle2009}.
        A similar conjecture holds for several part partitions,
        see \cite[Conjecture 1.2]{Lassalle2009}.
\end{itemize}
\begin{remark}
    This facts explain our changes of normalization.
    In Lassalle's work, the non-negativity of the coefficients is hidden:
    he has to use two variables $\alpha$ and $\beta=1-\alpha$ and choose
    a ``natural'' way to write all quantities in terms of $\alpha$ and $\beta$.
    Using our normalizations and the parameter $\gamma$,
    the non-negativity of the coefficients appears directly.
\end{remark}

\section{Polynomiality}
\label{SectPolynomial}

\comment{
\subsection{Notations}
As in the previous sections, most of our objects are indexed by partitions of integers.
Therefore it will be useful to use some short notations for small modifications
(adding or removing a box or a part) of partitions.
We denote by $\mu \cup (r)$ (resp.~$\mu \setminus (r)$)
the partition obtained from $\mu$ by adding  (resp.~deleting)
one part equal to $r$.
We denote by $\mu_{\downarrow r}=\mu \setminus (r) \cup (r-1)$ the partition
obtained by removing one box in a row of size $r$.
The reader might wonder what $\mu \setminus (r)$ and $\mu_{\downarrow r}$ mean
if $\mu$ does not have a part equal to $r$:
we will not use these notations in this context.
Finally, if $o$ is an outer corner of $\lambda$,
we denote by $\lambda^{(o)}$ the diagram obtained from $\lambda$ by adding a box
at place $o$.
}

\subsection{Main result}\label{SubsectMainResult}
\begin{theorem}\label{ThMain}
    The coefficient of $M_\rho$ in
    Jack character polynomial $L_\mu$ is a polynomial in $\gamma$
    of degree smaller or equal to 
    \[ \min \big( |\mu|+\ell(\mu) - |\rho|, |\mu|-\ell(\mu) - (|\rho|-2\ell(\rho)) \big).\]
    Moreover, it has the same parity as the integer  $|\mu|+\ell(\mu) - |\rho|$.
    
    The same is true for the coefficient of $R_\rho$ in $K_\mu$.
\end{theorem}

We do not prove this theorem in this extended abstract.
The proof is of course available in the long version of the paper \cite[Section 3]{FullPaper}.
Here, we are going to present a guidelines of this proof.

The difficulty is that the proof of the existence of the polynomials $L_\mu$ and
$K_\mu$ (Proposition \ref{PropExistenceJackKerov}) is not constructive.
However, M. Lassalle gives an algorithm to compute the polynomial $K_\mu$ \cite[Section 9]{Lassalle2009}, but his algorithm involves an induction on the size of the partition $|\mu|$.
The coefficients of $K_\mu$ are the solutions of an overdetermined linear system
involving the coefficients of $K_{\mu'}$, for some partitions $\mu'$ with $|\mu'|<|\mu|$.
His algorithm can be easily adapted to $L_\mu$ \cite[Section 3.2]{FullPaper}.

Our proof relies on this work and on the two following important facts:
\begin{itemize}[topsep=1pt, partopsep=1pt, itemsep=1pt, parsep=1pt]
\item the linear system computing the coefficients of $L_\mu$
contains a triangular subsystem (this is not true with $K_\mu$);
\item with our normalization of Jack characters and anisotropic moments,
the diagonal coefficients of this linear subsystem are independent of $\gamma$
(and hence invertible in $\QQ[\gamma]$).
\end{itemize}

The polynomiality in $\gamma$ follows from these two facts.
To obtain the bound on the degree, one has to look carefully at
the degrees of the coefficients of the linear system.\bigskip

Recall that our normalization is different from the one used by M. Lassalle.
After a simple rewriting game \cite[Section 3.6]{FullPaper},
we can see that Theorem \ref{ThMain} implies that the coefficients
of $L_\mu$ and $K_\mu$ with Lassalle's normalizations are polynomials
in $\alpha$ (that is the statement of Theorem \ref{ThmPol}). 

\subsection{Lapointe-Vinet theorem}\label{SubsectCoefJackPoly}
In this section, we prove that $\theta_\mu(\lambda)$ is a polynomial in $\alpha$.
This result was already known (see Introduction),
but in our opinion it illustrates
that Lassalle's approach to Jack polynomials is relevant.

To deduce this from the results above, one has to see how $M_k(\lambda)$
depends on $\alpha$.

\begin{lemma}
    Let $k \geq 2$ be an integer and $\lambda$ be a partition.
    Then $\sqrt{\alpha}^{k-2} M_k(\lambda)$ is a polynomial in
    $\alpha$ with integer coefficients.
    \label{LemMkPol}
\end{lemma}
\begin{proof}
    We use induction over $|\lambda|$.
    Let $o=(x,y)$ be an outer corner of $\lambda$,
we denote by $\lambda^{(o)}$ the diagram obtained from $\lambda$ by adding a box
at place $o$.
Comparing the corner of $\lambda^{(o)}$ and $\lambda$, we get that:
\[\OO_{\lambda^{(o)}} -\II_{\lambda^{(o)}} = 
\OO_{\lambda} - \II_{\lambda} + \{x-(y+1)\} + \{x+1-y\} - \{x-y\}\]
(for readers not used to $\lambda$-ring,
this equality can be understood as equality between formal sums of elements in the set).
After dilatation, we get
\[\OO_{T_{\sqrt{\alpha},\sqrt{\alpha}^{-1}} (\lambda^{(o)})} -
\II_{T_{\sqrt{\alpha},\sqrt{\alpha}^{-1}} (\lambda^{(o)})} =
 \OO_{T_{\sqrt{\alpha},\sqrt{\alpha}^{-1}} (\lambda)} - 
 \II_{T_{\sqrt{\alpha},\sqrt{\alpha}^{-1}} (\lambda)}
 +\{z_o-\frac{1}{\sqrt{\alpha}}\} +\{z_o+\alpha\} -\{z_o\}, \]
and $z_o=\sqrt{\alpha} x- y/\sqrt{\alpha}$ is the content of the
considered corner in $T_{\sqrt{\alpha},\sqrt{\alpha}^{-1}} (\lambda)$.

    By a standard $\lambda$-ring computations
    (see \cite[Proposition 8.3]{Lassalle2009}), this yields
        \[ M_k(\lambda^{(o)})-M_k(\lambda) =
        \sum_{\substack{r\geq 1, s,t \geq 0, \\ 2r+s+t \leq k}}
    z_o^{k-2r-s-t}
    \binom{k-t-1}{2r+s-1}\binom{r+s-1}{s}
    \left( -\gamma \right)^s M_t(\lambda),\]
    which can be rewritten as
    \begin{multline*} \sqrt{\alpha}^{k-2} M_k(\lambda^{(o)})
        -\sqrt{\alpha}^{k-2} M_k(\lambda) =
        \sum_{\substack{r\geq 1, s,t \geq 0, \\ 2r+s+t \leq k}}
        \alpha^r (\sqrt{\alpha} z_o)^{k-2r-s-t} \\
    \binom{k-t-1}{2r+s-1}\binom{r+s-1}{s}
    \left( \alpha-1 \right)^s \sqrt{\alpha}^{t-2} M_t(\lambda).
\end{multline*}
But $\sqrt{\alpha} z_o = \alpha x - y$ is a polynomial
in $\alpha$ with integer coefficients.
Thus the induction is immediate.
\end{proof}

Now we write, for $\mu,\lambda \vdash n$,
\[
    z_\mu \theta_\mu(\lambda)=\alpha^{\frac{|\mu|-\ell(\mu)}{2}} \Ch_\mu(\lambda)
= \sum_\rho \alpha^{\frac{|\mu|-\ell(\mu)-(|\rho|-2\ell(\rho))}{2}} 
a^\mu_\rho \left( \prod_{i \leq \ell(\rho)} \sqrt{\alpha}^{\rho_i-2}
M_{\rho_i}(\lambda) \right).
\]
The quantities $\alpha^{\frac{|\mu|-\ell(\mu)-(|\rho|-2\ell(\rho))}{2}} 
a^\mu_\rho$ and $\sqrt{\alpha}^{\rho_i-2}
M_{\rho_i}(\lambda)$ are polynomials in $\alpha$ (by Theorem~\ref{ThMain}
and Lemma~\ref{LemMkPol}), hence $\theta_\mu(\lambda)$ is a polynomial
in $\alpha$, which proves Theorem \ref{ThmCoefJackPoly}. 

\subsection{Gradation}
\label{SubsectGrad}
Looking at Theorem \ref{ThMain} it makes natural to consider
some gradations on $\Pola$.
This structure will also be useful in the next section.\medskip

The ring $\Pola$ of $\alpha$-polynomial functions can be endowed with a
gradation by deciding that $M_k$ is a homogeneous function of degree $k$:
as $(M_k)_{k \ge 2}$ is an algebraic basis of $\Pola$,
any choice of degree for $M_k$ (for all $k\ge 2$)
defines uniquely a gradation on $\Pola$.
Then $R_k$ is also a homogeneous function of degree $k$,
thanks to the moment-free cumulant relations,
see {\it e.g.} \cite[Section 2.4]{Biane1998}.

Theorem \ref{ThMain} shows that $\Ch_\mu$ has at most degree $|\mu|+\ell(\mu)$
(this has also been proved by M.~Lassalle \cite[Proposition 9.2 (ii)]{Lassalle2009}).
Note that $\Ch_\mu$ is not homogeneous in general.
Moreover, its component of degree $|\mu|+\ell(\mu)$ does not depend on $\alpha$.
As this dominant term is known in the case $\alpha=1$
(see for example \cite[Theorem 4.9]{SniadyGenusExpansion}),
one obtains the following result (which extends \cite[Theorem 10.2]{Lassalle2009}):
\[ \Ch_\mu = \prod_{i=1}^{\ell(\mu)} R_{\mu_i+1} + \text{smaller degree terms}.\]
In particular, $\Ch_\mu$ has exactly degree $|\mu|+\ell(\mu)$.

Consider the subspace $V_{\le d} \subset \Pola$ of elements of degree less or 
equal to $d$.
Its dimension is the number of partitions $\rho$ of size less or equal to $d$ with no parts equal to $1$.
By removing $1$ from every part of $\rho$, we see that this is also the number
of partitions $\mu$ such that $|\mu|+\ell(\mu) \le d$.
But the latter index the functions $\Ch_\mu$ lying in $V_{\le d}$.
Hence,
\[ V_{\le d}= \Vect \big(\{\Ch_\mu, |\mu|+\ell(\mu) \le d\} \big) \]
and the degree of an element in $\Pola$ can be determined as follows:
\begin{equation}\label{EqDegCombLinCh}
\deg\left( \sum_\mu a_\mu \Ch_\mu \right) = 
\max_{\mu : a_\mu \neq 0} |\mu|+\ell(\mu). 
\end{equation}
\begin{remark}
    The algebra $\Pola$ admits other relevant gradations,
    see \cite[Sections 3.5 and 3.8]{FullPaper}.
\end{remark}

\section{Application: asymptotics of large Young diagrams
}\label{SectRandomYD}

We consider the following deformation of the Plancherel measure
$ \PP_n^{(\alpha)} ( \lambda ) = \frac{\alpha^n n!}{j_\lambda^{(\alpha)}},$
where $j_\lambda^{(\alpha)}$ is the following deformation of the square of the
hook products:
\begin{equation}\label{EqDefAPlancherel}
j_\lambda^{(\alpha)} = \prod_{\Box \in \lambda} 
(\alpha a(\Box) + \ell(\Box)+1) 
(\alpha a(\Box) + \ell(\Box)+\alpha).
\end{equation}
Here, $a(\Box)$ and $\ell(\Box)$ are respectively the arm and leg length
of the box as defined in \cite[Chapter I]{Macdonald1995}.
The probability measure $\PP_n^{(\alpha)}$ on Young diagrams of size $n$
has appeared recently in
several research papers \cite{
BorodinOlshanskiJackPlancherel,
FulmanFluctuationChA2,
OlshanskiJackPlancherel,MatsumotoOddJM}
and is presented as an important area of research in Okounkov's survey on 
random partitions \cite[§ 3.3]{OkounkovRandomPartitions}.
When $\alpha=1$, it specializes to the well-known Plancherel measure
for the symmetric groups.

The following property,
which corresponds to the case
$\pi=(1^n)$ in \cite[Equation (8.4)]{MatsumotoOddJM},
characterizes the Jack-Plancherel measure: 
\[\esper_{\PP_n^{(\alpha)}}( \theta_\mu^{(\alpha)}(\lambda) ) =
\delta_{\mu,1^n},\]
where $\lambda$ is a random variable distributed according to
$\PP_n^{(\alpha)}$.

Using the definition of $\Ch_\mu$ we have:
\[\esper_{\PP_n^{(\alpha)}}(\Ch_\mu)=\begin{cases}
    n(n-1)\cdots(n-k+1) & \text{if }\mu=1^k \text{ for some }k \leq n,\\
    0   & \text{otherwise.}
\end{cases}\]
As $\Ch_\mu$ is a linear basis of $\Pola$, it implies the following lemma
(which is an analogue of \cite[Theorem 5.5]{OlshanskiJackPlancherel} with
another gradation).
\begin{lemma}
    Let $F$ be an $\alpha$-polynomial function. Then
$\esper_{\PP_n^{(\alpha)}}(F)$
    is a polynomial in $n$ of degree at most $\deg(F)/2$.
    \label{LemDegEsper}
\end{lemma}
\begin{proof}
    It is enough to verify this lemma on the basis $\Ch_\mu$ because
    of equation \eqref{EqDegCombLinCh}.
    But in this case $\esper_{\PP_n^{(\alpha)}}(F)$
    is explicit (see formula above) and the lemma is obvious
    (recall that $\deg(\Ch_\mu)=|\mu|+\ell(\mu)$; 
    see Section \ref{SubsectGrad}).
\end{proof}

Let $(\lambda^n)_{n \geq 1}$ be a sequence of random partitions, where
$\lambda^n$ has distribution $\PP_n^{(\alpha)}$.
In the case $\alpha=1$, it has been proved in 1977 separately by
Logan and Shepp \cite{LoganShepp1977} and Kerov and Vershik
\cite{KerovVershikLimitYD}
that, in probability, 
\begin{equation}\label{EqConvUnifPlancherelA1}
    \lim_{n \to \infty} \left\lVert \omega(D_{1/\sqrt{n}}(\lambda^n)) - \Omega \right\rVert = 0,
\end{equation}
where $\Omega$ is the limit shape given explicitly as follows:
\[\Omega(x) = \begin{cases}
    |x| & \text{if }|x| \geq 2;\\
    \frac{2}{\pi} \left( x \cdot \arcsin \frac{x}{2} + \sqrt{4-x^2} \right) 
    & \text{otherwise.}
\end{cases}\]
Recall from Section \ref{SubsectGeneralizedYD}
that $D_s(\lambda)$ is the Young diagram $\lambda$ {\em dilated} by a factor $s$
and $\omega(\lambda)$ is by definition the function whose graphical representation
is the border of $\lambda$, rotated by $45\degree$ (see Section \ref{SubsectGeneralizedYD})
and stretched by $\sqrt{2}$.

In the general $\alpha$ case, we have the following weak convergence result:

\begin{proposition} \label{PropAPlanchWeakConv}
    For any 1-polynomial function $F \in \Pol^{(1)}$, when $n$ tends to infinity, one has
    \[ F\big(T_{\sqrt{\alpha/n},1/\sqrt{n \alpha}} (\lambda^n)\big) \xrightarrow{\PP_n^{(\alpha)}} 
    F( \Omega) , \]
    where $\xrightarrow{\PP_n^{(\alpha)}}$ means convergence in probability.
\end{proposition}

\begin{proof}
    As $(R^{(1)}_k)_{k \geq 2}$ is an algebraic basis of $\Pol^{(1)}$,
    it is enough to prove the proposition for any $R^{(1)}_k$.

Let $\mu$ be partition.
As mentioned at the beginning of the section, one has:
\begin{equation}\label{EqTopDeg1}
    \prod_{i \leq \ell(\mu)} R^{(\alpha)}_{\mu_i+1} =
    \Ch_\mu + \text{ terms of degree at most }|\mu|+\ell(\mu)-1.
\end{equation}
Together with Lemma~\ref{LemDegEsper} and the formula for $\esper_{\PP_n^{(\alpha)}}(\Ch_\mu)$,
this implies:
\[ \esper_{\PP_n^{(\alpha)}} \left(
\prod_{i \leq \ell(\mu)} R^{(\alpha)}_{\mu_i+1}\right) = \begin{cases}
    n(n-1)\cdots(n-k+1) + O(n^{k-1}) & \text{if }\mu=1^k \text{ for some }k; \\
    o(n^{\frac{|\mu|+\ell(\mu)}{2}}) & \text{otherwise.}
\end{cases}\]
In particular, one has that
\begin{align*}
    \esper_{\PP_n^{(\alpha)}} \big(
        R^{(\alpha)}_k(D_{1/\sqrt{n}}(\lambda^n))\big) &=
\frac{1}{n^{k/2}}
    \esper_{\PP_n^{(\alpha)}} (R^{(\alpha)}_k) = \delta_{k,2} + O\left(\frac{1}{\sqrt{n}}\right); \\
    \Var_{\PP_n^{(\alpha)}} \big(R^{(\alpha)}_k(D_{1/\sqrt{n}}(\lambda^n))\big) &= \frac{1}{n^k}
\left( 
    \esper_{\PP_n^{(\alpha)}} \big((R^{(\alpha)}_k)^2\big) - \esper_{\PP_n^{(\alpha)}} (R^{(\alpha)}_k)^2
\right)
    = O\left(\frac{1}{n}\right).
\end{align*}
Thus, for each $k$, $R^{(\alpha)}_k(D_{1/\sqrt{n}}(\lambda^n))$ converges in probability
towards $\delta_{k,2}$. 
But, by definition 
\[R^{(\alpha)}_k(D_{1/\sqrt{n}}(\lambda^n))=R^{(1)}_k \big(T_{\sqrt{\alpha/n},1/\sqrt{n \alpha}} (\lambda^n)\big)\]
and $(\delta_{k,2})_{k \geq 2}$ is the sequence of free cumulants
of the continuous diagram $\Omega$
(see \cite[Section 3.1]{Biane2001}),
\emph{i.e.} 
$\delta_{k,2} = R_k^{(1)}(\Omega).
$
\end{proof}

Roughly speaking, Proposition \ref{PropAPlanchWeakConv} means that,
the stretched Young diagram 
$T_{\sqrt{\alpha/n},1/\sqrt{n \alpha}} (\lambda^n) $ converges weakly
towards $\Omega$ (in probability).
So this result already means that the considered diagrams admit some
limit shape.

However, 
it would be desirable to obtain a result with uniform convergence,
which is a more natural notion of convergence.
This can be done thanks to the following lemma.
\begin{lemma}\label{LemFiniteSupport}
There exists a constant $C$ such that
\[\lim_{n \to \infty} \PP \left[
 \max\left( \frac{c(\lambda^n)}{\sqrt{n}};\frac{r(\lambda^n)}{\sqrt{n}}\right)
 \le C
 \right] =1, \]
 where, for each $n$, the diagram $\lambda^n$ is chosen randomly 
 with distribution $\PPna$ and
 $r(\lambda^n)$ and $c(\lambda^n)$ are respectively
 its numbers of rows and columns.
\end{lemma}

The proof of this lemma is quite technical and relies on the explicit formula 
\eqref{EqDefAPlancherel} for $\PPna$.
It can be found in \cite[Section 6.4]{FullPaper}.
We can now state the uniform convergence result.

\begin{theorem}
For each $n$, let $\lambda^n$ be a random Young diagram of size $n$
 distributed with $\alpha$-Plancherel measure. Then, in probability,
\[ \lim_{n \to \infty} \left\lVert \omega\big(T_{\sqrt{\alpha/n},1/\sqrt{n \alpha}} (\lambda^n)\big)
 - \Omega \right\rVert = 0.\]
\end{theorem}

\begin{proof}
    It follows from Proposition \ref{PropAPlanchWeakConv} and 
    Lemma \ref{LemFiniteSupport}
    by the same argument as the one given in
    \cite[Theorem 5.5]{IvanovOlshanski2002}.
\end{proof}

The idea of using polynomial functions to study the asymptotic shape of Young
diagrams
has been developped by S.~Kerov (see \cite{IvanovOlshanski2002}).
In the case $\alpha=1$, he gave more precise result that what we have here:
he proved that for any polynomial function $F$, the quantity $F(\lambda^n)$ has
Gaussian fluctuations.
A better understanding of polynomials $K_\mu$ could lead to a proof
of a similar phenomena in the general $\alpha$ case,
using the ideas introduced in \cite{'Sniady2006c}.
Let us mention the existence of a partial result
(corresponding to $F=\Ch^{(2)}$) obtained
by J.~Fulman \cite[Theorem 1.2]{FulmanFluctuationChA2} by another method.

{\small
\bibliographystyle{abbrv}

\bibliography{biblio2011}
}
\end{document}